\documentclass{birkjour}

\usepackage{amssymb,amsmath,newlfont,enumerate}

\newtheorem{theorem}{Theorem}[section]

\newtheorem{question}[theorem]{Question}
 
\theoremstyle{definition}

\theoremstyle{remark}
\newtheorem*{remark}{Remark}
\newtheorem*{remarks}{Remarks}

\numberwithin{equation}{section}

\newcommand{\CC}{{\mathbb C}}
\newcommand{\DD}{{\mathbb D}}
\newcommand{\TT}{{\mathbb T}}

\newcommand{\cD}{{\cal D}}
\newcommand{\cH}{{\cal H}}

\DeclareMathOperator{\hol}{Hol}

\begin{document}

\title[Dirichlet spaces with superharmonic weights]
 {Dirichlet spaces with superharmonic weights\\ and de Branges--Rovnyak spaces}

\author[El-Fallah]{O. El-Fallah}
\address{Laboratoire Analyse et Applications (URAC/03), Universit\'e Mohamed V,\\
B.P. 1014 Rabat, Morocco}
\email{elfallah@fsr.ac.ma}

\thanks{El-Fallah supported by  CNRST URAC/03 \&
Acad\'emie Hassan\! II des sciences et techniques.\\
Kellay supported by UMI-CRM.\\
Mashreghi supported by NSERC.\\
Ransford supported by NSERC and the Canada research chairs program}

\author[Kellay]{K. Kellay}
\address{Institut de Math\'ematiques de Bordeaux, Universit\'e de Bordeaux,\\
351 cours de la Lib\'eration, 33405 Talence Cedex, France}
\email{karim.kellay@math.u-bordeaux1.fr}

\author[Klaja]{H. Klaja}
\address{D\'epartement de math\'ematiques et de statistique, Universit\'e Laval,\\  
1045 avenue de la M\'edecine, Qu\'ebec (QC),  Canada G1V 0A6}
\email{hubert.klaja@gmail.com}

\author[Mashreghi]{J. Mashreghi}
\address{D\'epartement de math\'ematiques et de statistique, Universit\'e Laval,\\  
1045 avenue de la M\'edecine, Qu\'ebec (QC),  Canada G1V 0A6}
\email{javad.mashreghi@mat.ulaval.ca}

\author[Ransford]{T. Ransford}
\address{D\'epartement de math\'ematiques et de statistique, Universit\'e Laval,\\  
1045 avenue de la M\'edecine, Qu\'ebec (QC), Canada G1V 0A6}
\email{ransford@mat.ulaval.ca}



\date{10 March 2015}

\begin{abstract}
We consider Dirichlet spaces with superharmonic weights.
This class contains both the harmonic weights and the power weights.
Our main result is a characterization of the Dirichlet spaces with superharmonic weights 
that can be identified as  de Branges--Rovnyak spaces. 
As an application, we obtain the dilation inequality
\[
\cD_\omega(f_r)\le \frac{2r}{1+r}\cD_\omega(f) \qquad(0\le r<1),
\]
where $\cD_\omega$ denotes the Dirichlet integral with superharmonic weight $\omega$,
and $f_r(z):=f(rz)$ is the $r$-dilation of the holomorphic function $f$.
\end{abstract}

\maketitle

\section{Introduction}\label{S:intro}

Let $\DD$ be the open unit disk, 
let $dA$ be normalized area measure on $\DD$, 
and let $\hol(\DD)$ be the space of holomorphic functions on $\DD$.
Given a non-negative function $\omega\in  L^1(\DD)$ and $f\in\hol(\DD)$, 
we define
\[
\cD_\omega(f):=\int_\DD|f'(z)|^2\,\omega(z)\,dA(z).
\]
The \emph{weighted Dirichlet space} $\cD_\omega$ 
is the set of $f\in\hol(\DD)$ with $\cD_\omega(f)<\infty$.
Obviously, if $\omega\equiv1$, 
then $\cD_\omega$ is just the classical Dirichlet space $\cD$.

In this article we shall be mainly interested in the case 
where $\omega$ is a superharmonic weight. 
This class of weights was introduced  by Aleman \cite{Al93}. 
For such $\omega$,  we automatically have $\omega\in L^1(\DD)$, and,
provided that $\omega\not\equiv0$,
we also have $\cD_\omega\subset H^2$, the Hardy space.
It is customary to define
\[
\|f\|_{\cD_\omega}^2:=\|f\|_{H^2}^2+\cD_\omega(f) 
\qquad(f\in\cD_\omega),
\]
making $\cD_\omega$ a Hilbert space.

The class of superharmonic weights includes two important subclasses:
\begin{itemize}
\item the power weights
$\omega(z):=(1-|z|^2)^\alpha ~(0\le\alpha\le 1)$,
which form a scale linking the classical Dirichlet space ($\alpha=0$) 
to the Hardy space ($\alpha=1$).
\item the harmonic weights, introduced by Richter \cite{Ri91} in connection with 
his analysis of shift-invariant subspaces of the classical Dirichlet space.
\end{itemize}

Superharmonic weights have the important property that 
the dilates of $f_r(z):=f(rz)$ of each function $f\in\cD_\omega$ satisfy
\begin{equation}\label{E:frineq}
\cD_\omega(f_r)\le C\cD_\omega(f)
\qquad(0\le r<1),
\end{equation}
where $C$ is an absolute constant. 
From this, it is not hard to deduce that
$\|f_r-f\|_{\cD_\omega}\to0$ as $r\to1^-$, 
and that polynomials are dense in $\cD_\omega$.

Inequality \eqref{E:frineq} was  first proved by Richter and Sundberg \cite{RS91} 
in the case where $\omega$ is harmonic, with $C=4$. 
It was generalized to superharmonic weights by Aleman \cite{Al93}, 
with an improved  constant $C=5/2$. 
For harmonic weights, 
the constant was further improved to  $C=1$ by Sarason \cite{Sa97}.

The basis of Sarason's method was to show that, 
if $\omega=P_\zeta$ (the Poisson kernel at a point $\zeta\in\TT$),
then it is possible to identify $\cD_\omega$ as 
a de Branges--Rovnyak space $\cH(b)$, the norms being identical
(we shall define $\cH(b)$ later). 
Working within $\cH(b)$, 
one can deduce that \eqref{E:frineq} holds with $C=1$, 
at least for these special $\omega$.
Inequality \eqref{E:frineq} for  general harmonic $\omega$  
then follows easily by using an averaging argument.

It turns out that Sarason's construction has a sort of converse. 
The only harmonic weights $\omega$ for which 
$\cD_\omega$ can be identified as a de Branges--Rovnyak space 
are multiples of $P_\zeta$, where $\zeta\in\TT$. 
This was proved in \cite{CGR10}. 
(One can also study a weaker notion of `identified', 
where the spaces $\cD_\omega$ and $\cH(b)$ are allowed to carry different norms. 
This appears to be more subtle: see \cite{CR13}.)

Our purpose in this paper is twofold. 
First, we shall extend Sarason's result 
by exhibiting a new family of superharmonic weights $\omega$
for which $\cD_\omega$ can be identified as a de Branges--Rovnyak space $\cH(b)$. 
We shall deduce from this that
the inequality \eqref{E:frineq} holds  with $C=1$ for all superharmonic weights. 
Second, we shall prove a converse result, 
characterizing those superharmonic weights $\omega$ 
for which $\cD_\omega$ is equal to an $\cH(b)$.
We shall also consider what happens for more general weights.

\section{Superharmonic weights}

Let $\omega$ be a positive superharmonic function on $\DD$.

By standard results from potential theory,
 $\omega$ is locally integrable on $\DD$,
and $(1/r^2)\int_{|z|\le r}\omega\,dA$ is 
a decreasing function of $r$ for $0<r<1$. 
It follows that $\omega\in L^1(\DD)$,
so it is admissible as a Dirichlet weight.

Now suppose that $\omega\not\equiv0$. 
We shall show that $\cD_\omega\subset H^2$.
By the minimum principle $\omega(z)>0$ for all $z\in\DD$ 
(possibly infinite at some points).
By lower semicontinuity, 
$m:=\inf_{|z|\le1/e}\omega(z)$ is attained and therefore $m>0$.
By the minimum principle again, 
$\omega(z)\ge m\log(1/|z|)$ for all $z$ with $1/e<|z|<1$.
Since the Dirichlet space with weight $\log(1/|z|)$ is just $H^2$,
it follows easily that $\cD_\omega\subset H^2$, as claimed.

It thus makes sense to define the norm 
$\|\cdot\|_{\cD_\omega}$ on $\cD_\omega$ by
\begin{equation}\label{E:Dnormdef}
\|f\|_{\cD_\omega}^2:=\|f\|_{H^2}^2+\cD_\omega(f) 
\qquad(f\in\cD_\omega).
\end{equation}
With respect to this norm, $\cD_\omega$ is a Hilbert space.

We shall make extensive use of the following representation formula
for positive superharmonic functions $\omega$ on $\DD$.
Given such an $\omega$,
there exists a unique positive finite Borel measure $\mu$ 
on $\overline{\DD}$ such that, for all $z\in\DD$,
\begin{equation}\label{E:murep}
\omega(z)
=\int_\DD \log\Bigl|\frac{1-\overline{\zeta}z}{\zeta-z}\Bigr|\frac{2}{1-|\zeta|^2}\,d\mu(\zeta)
+\int_\TT\frac{1-|z|^2}{|\zeta-z|^2}\,d\mu(\zeta).
\end{equation}
(This is  the usual decomposition of $\omega$ as a potential plus a harmonic function.)
We then write $\cD_\mu$ for $\cD_\omega$, 
and further $\cD_\zeta$ for $\cD_{\delta_\zeta}$. 
Thus, for $f\in\hol(\DD)$, 
\[
\cD_\zeta(f)=
\begin{cases}
\displaystyle\int_\DD \log\Bigl|\frac{1-\overline{\zeta}z}{\zeta-z}\Bigr|\frac{2}{1-|\zeta|^2}|f'(z)|^2\,dA(z), 
&\zeta\in\DD,\\
\displaystyle\int_\DD\frac{1-|z|^2}{|\zeta-z|^2}|f'(z)|^2\,dA(z), 
&\zeta\in\TT.
\end{cases}
\]
For general $\mu$, we can recover $\cD_\mu(f)$  from $\cD_\zeta(f)$ via Fubini's theorem: 
\begin{equation}\label{E:average}
\cD_\mu(f)=\int_{\overline{\DD}}\cD_\zeta(f)\,d\mu(\zeta).
\end{equation}

Finally, we shall need the following Douglas-type formula for $\cD_\zeta(f)$.

\begin{theorem}\label{T:Douglas}
Let $f\in H^2$ and let $\zeta\in\overline{\DD}$. 
If $\zeta\in\DD$, then
\begin{equation}\label{E:Douglas}
\cD_\zeta(f)=
\frac{1}{2\pi}\int_\TT \Bigl|\frac{f(\lambda)-f(\zeta)}{\lambda-\zeta}\Bigr|^2\,|d\lambda|.
\end{equation}
If $\zeta\in\TT$ and $f(\zeta):=\lim_{r\to1^-}f(r\zeta)$ exists, 
then the same formula holds.
Otherwise $\cD_\zeta(f)=\infty$.
\end{theorem}

For a proof, and more background on superharmonic weights, see \cite{Al93}.

\section{De Branges--Rovnyak spaces}\label{S:deBR}

The de Branges--Rovnyak spaces  
are a family of (not necessarily closed) subspaces $\cH(b)$ of $H^2$,
para\-metrized by elements $b$ of the closed unit ball of $H^\infty$.
They  were introduced by de Branges and Rovnyak 
in the appendix of \cite{dBR66a} 
and further studied in  \cite{dBR66b}.
For background information we refer to the books of 
de Branges and Rovnyak \cite{dBR66b}, Sarason \cite{Sa94},
and the forthcoming monograph of Fricain and Mashreghi \cite{FM15}.

Given $\psi\in L^\infty(\TT)$, we define
the  Toeplitz operator $T_\psi:H^2\to H^2$ by
\[
T_\psi f:=P_+(\psi f) \qquad(f\in H^2),
\]
where $P_+:L^2(\TT)\to H^2$ is 
the orthogonal projection of $L^2(\TT)$ onto $H^2$. 
Clearly $T_\psi$ is a bounded operator on $H^2$,
and its adjoint is $T_{\overline{\psi}}$.

Given $b\in H^\infty$
with $\|b\|_{H^\infty}\le1$,
the \emph{de Branges--Rovnyak space}
$\cH(b)$ is defined as the image of $H^2$ 
under the operator $(I-T_ bT_{\overline b})^{1/2}$.
We put a norm on $\cH(b)$ making $(I-T_bT_{\overline{b}})^{1/2}$ 
a partial isometry from $H^2$ onto $\cH(b)$, namely
\[
\|(I-T_bT_{\overline{b}})^{1/2}f\|_{\cH(b)}:=\|f\|_{H^2}
\qquad(f\in H^2\ominus \ker(I-T_bT_{\overline{b}})^{1/2})).
\]

We have defined $\cH(b)$ as in Sarason's book \cite{Sa94}. 
The original definition of de Branges and Rovnyak,
based on the notion of complementary space, 
is different but equivalent. 
An explanation of the equivalence can be found in \cite[pp.7--8]{Sa94}. 
A third approach is to start from the positive kernel
\[
k_w^b(z):=\frac{1-\overline{b(w)}b(z)}{1-\overline{w}z}
\qquad(z,w\in\DD),
\]
and to define $\cH(b)$ as 
the reproducing kernel Hilbert space associated with this kernel.

The theory of $\cH(b)$-spaces is pervaded by a fundamental dichotomy,
namely whether $b$ is or is not an extreme  point 
of the closed unit ball of $H^\infty$. 
This is illustrated by following result.

\begin{theorem}\label{T:nonextreme}
Let $b\in H^\infty$ with $\|b\|_{H^\infty}\le 1$. 
The following  are equivalent:
\begin{itemize}
\item[\rm(i)] $b$ is a non-extreme point of the closed unit ball of $H^\infty$;
\item[\rm(ii)] $\log (1-|b|^2)\in L^1(\TT)$;
\item[\rm(iii)]  $\cH(b)$ contains 
all functions holomorphic in a neighborhood of $\overline{\DD}$.
\end{itemize}
Furthermore, if $b$ is non-extreme, then the polynomials are dense in $\cH(b)$.
\end{theorem}

\begin{proof}
The equivalence between (i) and (ii) is proved in \cite[Theorem~7.9]{Du00}. 
That (i) implies (iii) is proved in \cite[\S IV-6]{Sa94},
and that (iii) implies (i) follows from \cite[\S V-1]{Sa94}.
Finally, the density of polynomials when $b$ is non-extreme
is proved in \cite[\S IV-3]{Sa94}; 
a constructive proof of this result is given in \cite{EFKMR15}.
\end{proof}

Henceforth we shall simply say that $b$ is `extreme' or `non-extreme', 
it being understood that this relative to the closed unit ball of $H^\infty$.

From the equivalence between (i) and (ii), it follows that, 
if $b$ is non-extreme,
then there exists a unique  outer function $a$ 
such that $a(0)>0$ and  $|a|^2+|b|^2=1$ a.e.\ on $\TT$ 
(see \cite[\S IV-1]{Sa94}). 
We shall call $(b,a)$ a \emph{pair}. 
The following result gives 
a useful characterization of $\cH(b)$ in this case.

\begin{theorem}[\protect{\cite[\S IV-1]{Sa94}}]\label{T:f+}
Let $b$ be non-extreme, 
let $(b,a)$ be a pair and let $f\in H^2$.
Then $f\in\cH(b)$ if and only if 
$T_{\overline b}f\in T_{\overline a}(H^2)$.
In this case, there exists a unique function $f^+\in H^2$
such that $T_{\overline b}f=T_{\overline a}f^+$, and
\begin{equation}\label{E:f+}
\|f\|_{\cH(b)}^2=\|f\|_{H^2}^2+\|f^+\|_{H^2}^2.
\end{equation}
\end{theorem}

Given a pair $(b,a)$, 
the function $\phi:=b/a$ is the quotient of two functions in $H^\infty$, 
the denominator being outer.
In other words, $\phi\in N^+$, the Smirnov class. 
Conversely, given $\phi\in N^+$, we can write $\phi=b/a$, 
where $a,b\in H^\infty$ and $a$ is outer. 
Multiplying $a$ and $b$ by an appropriately chosen outer function, 
we may further ensure that $|a|^2+|b|^2=1$ a.e.\ on $\TT$ 
and that $a(0)>0$, 
in other words, that $(b,a)$ is a pair.
Then $a$ and $b$ are uniquely determined.
There is thus a bijection $b\leftrightarrow \phi$ 
between non-extreme functions $b$ 
and elements $\phi$ of the Smirnov class. 
Note that $b$ and $\phi$ have the same inner factor.
Also $\phi$ is bounded if and only if $\|b\|_{H^\infty}<1$. 
In this case, $(I-T_bT_{\overline{b}})$ is an invertible operator on $H^2$, 
and consequently $\cH(b)=H^2$ as vector spaces, 
although the norms may be different.

\section{Extension of Sarason's theorem}

Our goal in this section is to establish the following theorem 
and examine one of its consequences.

\begin{theorem}\label{T:Dmu=Hb}
Let $\zeta\in\overline{\DD}$, 
and let $b\in H^\infty$ be the function 
corresponding to $\phi(z):=z/(1-\overline{\zeta}z)$.
Then $\cD_\zeta=\cH(b)$ with equality of norms.
\end{theorem}

\begin{remarks}
(1) This theorem extends Sarason's result \cite{Sa97}, 
which is the special case $\zeta\in\TT$.

(2) A computation shows that, 
with $\phi(z)=z/(1-\overline{\zeta}z)$, we have
\[
b(z)=\frac{z}{A-Bz}
\qquad\text{and}\qquad
a(z)=\frac{1-\overline{\zeta}z}{A-Bz},
\]
where
\[
A:=\Bigl(\frac{2+|\zeta|^2+\sqrt{4+|\zeta|^4}}{2}\Bigr)^{1/2}
\quad\text{and}\quad
B:=\Bigl(\frac{2+|\zeta|^2-\sqrt{4+|\zeta|^4}}{2}\Bigr)^{1/2}\frac{\overline{\zeta}}{|\zeta|}.
\]
However, we do not need the precise formulas for $b$ and $a$ in what follows.
\end{remarks}

\begin{proof}
As the result is already known for $\zeta\in\TT$, 
we shall concentrate on the case $\zeta\in\DD$.
In this case,  $\cD_\zeta$ and $\cH_b$ are both equal to $H^2$ as vector spaces,
with equivalent norms.
For $\cD_\zeta$ this follows from Theorem~\ref{T:Douglas},
and for $\cH(b)$ we already observed this to be true whenever $\phi$ is bounded.
The content of the theorem is that 
the norms on $\cD_\zeta$ and $\cH(b)$ are in fact identical.

For $w\in\DD$, let $k_w$ denote the Cauchy kernel, namely
\[
k_w(z):=\frac{1}{1-\overline{w}z}
\qquad(z\in\DD).
\]
It is enough to prove that 
$\langle f,g\rangle_{\cD_\zeta}=\langle f,g\rangle_{\cH(b)}$
when $f$ and $g$ are finite linear combinations of Cauchy kernels, 
since such $f,g$ are dense in $H^2$.
By sesquilinearity, this reduces to checking that 
$\langle k_{w_1},k_{w_2}\rangle_{\cD_\zeta}=\langle k_{w_1},k_{w_2}\rangle_{\cH(b)}$ 
for all $w_1,w_2\in\DD$. As both sides of this last equation 
are holomorphic in $w_2$ and antiholomorphic in $w_1$,
it is sufficient to prove it in the case when $w_1=w_2$.
Thus we need to show that 
$\|k_w\|_{\cD_\zeta}^2=\|k_w\|_{\cH(b)}^2$ for all $w\in\DD$.
By \eqref{E:Dnormdef} and \eqref{E:f+}, 
this amounts to verifying that 
$\cD_\zeta(k_w)=\|k_w^+\|_{H^2}^2$ for all $w\in\DD$,
which we now proceed to do.

On the one hand, by Theorem~\ref{T:Douglas}, we have
\[
\cD_\zeta(k_w)
=\frac{1}{2\pi}\int_\TT \Bigl|\frac{k_w(\lambda)-k_w(\zeta)}{\lambda-\zeta}\Bigr|^2\,|d\lambda|
=\Bigl\|\frac{\overline{w}}{1-\overline{w}\zeta}k_w\Bigr\|_{H^2}^2.
\]
The Cauchy kernel is the reproducing kernel for $H^2$, 
so $\langle f,k_w\rangle_{H^2}=f(w)$ for all $f\in H^2$, 
and in particular $\|k_w\|_{H^2}^2=k_w(w)=1/(1-|w|^2)$. 
Hence
\[
\cD_\zeta(k_w)=\frac{|w|^2}{|1-\overline{w}\zeta|^2(1-|w|^2)}.
\]

On the other hand,
by a standard property of adjoints of multiplication operators acting on reproducing kernels,
we have $T_{\overline{h}}k_w=\overline{h(w)}k_w$ 
for all $h\in H^\infty$ and $w\in\DD$.
In particular  this is true when $h=b$ and when $h=a$, 
where $(b,a)$ is the pair with $b/a=\phi$.
In the notation of Theorem~\ref{T:f+}, 
it follows that $k_w^+=\overline{\phi(w)}k_w$, whence
\[
\|k_w^+\|_{H^2}^2=|\phi(w)|^2\|k_w\|_{H^2}^2
=\frac{|w|^2}{|1-\overline{w}\zeta|^2(1-|w|^2)}.
\]
Thus $\|k_w^+\|_{H^2}^2=\cD_\zeta(k_w)$ for all $w\in\DD$, as desired.
\end{proof}

Using the same basic idea as in \cite{Sa97},
we can use this theorem to deduce that \eqref{E:frineq} holds 
for all superharmonic weights $\omega$
with constant $C=1$. 
In fact, just as in \cite{Sa97}, we have the following even stronger result.

\begin{theorem}
Let $\mu$ be a finite positive measure on $\overline{\DD}$ 
and let $f\in\cD_\mu$. Then
\[
\cD_\mu(f_r)\le \frac{2r}{1+r}\cD_\mu(f) \qquad(0\le r< 1).
\]
\end{theorem}

\begin{proof}
We prove the result for $\mu=\delta_\zeta~(\zeta\in\overline{\DD})$. 
The general case follows by integrating up and using \eqref{E:average}.

By Theorem~\ref{T:Dmu=Hb}, 
we have $\cD_\zeta(f)=\|f^+\|_{\cH_b}^2$, 
where $(b,a)$ is the pair for which 
$\phi(z):=b(z)/a(z)=z/(1-\overline{\zeta}z)$. 
So we need to show that, with this choice of $\phi$, we have
\[
\|(f_r)^+\|_{H^2}^2\le \frac{2r}{1+r}\|f^+\|_{H^2}^2.
\]

Given $h\in H^2$, we have
\[
\langle (f_r)^+,ah\rangle_{H^2}
=\langle f_r,bh\rangle_{H^2}
=\langle f,b_rh_r\rangle_{H^2}
=\langle f^+,(\phi_r/\phi)a_rh_r\rangle_{H^2},
\]
and thus
\[
|\langle (f_r)^+,ah\rangle_{H^2}|
\le \|f^+\|_{H^2}\|\phi_r/\phi\|_{H^\infty}\|ah\|_{H^2}.
\]
As $a$ is an outer function, $aH^2$ is dense in $H^2$, and so 
\[
\|(f_r)^+\|_{H^2}\le\|\phi_r/\phi\|_{H^\infty}\|f^+\|_{H^2}.
\]
Finally, an elementary computation shows that
\[
\|\phi_r/\phi\|_{H^\infty}=r(1+|\zeta|)/(1+r|\zeta|)\le 2r/(1+r),
\]
whence the result.
\end{proof}

\section{A converse result}

Theorem~\ref{T:Dmu=Hb} furnishes a list of couples $(\mu,b)$ 
for which $\cD_\mu=\cH(b)$. 
In this section we prove a converse, which shows that,
apart from scalar multiples taken in a natural sense,
these are the only such couples. 
It generalizes a result of \cite{CGR10}, 
where it was proved in the case when
$\mu$ is a measure on $\TT$.

Note that we may assume from the outset that $b$ is 
a non-extreme point of the closed unit ball of $H^\infty$.
Indeed, if we are to have $\cD_\mu=\cH(b)$, then, 
since $\cD_\mu$ contains all functions holomorphic 
on a neighborhood of $\overline{\DD}$,
the same must be true of $\cH(b)$.  
By Theorem~\ref{T:nonextreme}, this entails that $b$ is non-extreme. 

\begin{theorem}\label{T:converse}
Let $\mu$ be a finite positive measure on $\overline{\DD}$ with $\mu\not\equiv0$.
Let $(b,a)$ be pair and let $\phi:=b/a$.
Then $\cD_\mu=\cH(b)$ with equality of norms
if and only if there exist 
$\zeta\in\overline{\DD}$ and $\alpha\in\CC\setminus\{0\}$ such that
\[
\mu=|\alpha|^2\delta_\zeta 
\qquad\text{and}\qquad 
\phi(z)=\alpha z/(1-\overline{\zeta}z).
\]
\end{theorem}

\begin{proof}
The `if' part follows directly from Theorem~\ref{T:Dmu=Hb}.
For the `only if', we observe that,
if we have equality of norms, 
then $\|f^+\|_{H^2}^2=\cD_\mu(f)$ for all $f$ in the space,
in particular for $f=k_w$, the Cauchy kernels. 
Performing similar calculations 
to those in the proof of Theorem~\ref{T:Dmu=Hb}, 
we obtain the relation
\begin{equation}\label{E:phieqn}
|\phi(w)|^2=|w|^2\int_{\overline{\DD}} \frac{d\mu(z)}{|1-z\overline{w}|^2} 
\qquad(w\in\DD).
\end{equation}
In particular $\phi(0)=0$. 
If we  write $\phi(w)/w$ as a Taylor series,
substitute it into the formula above, 
expand  in powers of $w$ and $\overline{w}$, 
and equate coefficients, 
then we end up with the following relation between the moments of $\mu$:
\begin{equation}\label{E:moment}
\mu(\overline{\DD})\int_{\overline{\DD}}z^n \overline{z}^{m}\,d\mu(z)
=\int_{\overline{\DD}}z^n \,d\mu(z)\int_{\overline{\DD}} \overline{z}^{m}\,d\mu(z)
\qquad(m,n\ge0).
\end{equation}
In particular, we have 
$\mu(\overline{\DD})(\int|z|^2\,d\mu)=(\int z\,d\mu)(\int\overline{z}\,d\mu)$.
This can be re-written as 
$\int|z-\zeta|^2\,d\mu(z)=0$, 
where $\zeta:=(\int z\,d\mu)/\mu(\overline{\DD})\in\overline{\DD}$.
It follows that $\mu=c\delta_\zeta$ for some $c>0$. 
Substituting this back into formula \eqref{E:phieqn}, we find that
$|\phi(w)|^2=c|w|^2/|1-\overline{\zeta}w|^2$,
and so $\phi(w)=\alpha w/(1-\overline{\zeta}w)$,
where $\alpha$ is a complex constant with $|\alpha|^2=c$.
This completes the proof.
\end{proof}

\section{More general weights}

We have identified those superharmonic weights $\omega$ for which 
$\cD_\omega$ is equal to a de Branges--Rovnyak space $\cH(b)$.
What about more  general weights? 
In this section we obtain some results in this direction,
leading to questions that we think are interesting in their own right.

We begin with a result about which functions $b$ can arise in this context.

\begin{theorem}\label{T:binner}
Let $b$ be in the closed unit ball of $H^\infty$. 
If there exists a weight $\omega$ 
with $0<\|\omega\|_{L^1(\DD)}<\infty$
such that $\cH(b)=\cD_\omega$, with equality of norms, 
then $b$ is non-extreme and its inner factor is exactly $z$.
Consequently the inner factor of the associated function $\phi$ is also $z$.
\end{theorem}

\begin{remark}
This result is in stark contrast with what happens 
if we do not insist on equality of norms.
See \cite[Theorem~4.5]{CR13}.
\end{remark}

\begin{proof}
That  $b$ is non-extreme is proved just as in 
the remark preceding the proof of Theorem~\ref{T:converse}.
We then have $\|f^+\|_{H^2}^2=\cD_\omega(f)$ for all $f\in \cH(b)$.
In particular  $f^+=0$ if and only if $\cD_\omega(f)=0$. Now,
\[
f^+=0 \iff T_{\overline{b}}f=0 \iff f\perp bH^2 \iff f\perp b_iH^2,
\]
where $b_i$ is the inner factor of $b$.
Also, since $\|\omega\|_{L^1(\DD)}>0$,
we have 
\[
\cD_\omega(f)=0 \iff f \text{~is a constant~}\iff f\perp zH^2.
\]
Combining these remarks, we conclude that $b_iH^2=zH^2$, 
whence $b_i=z$.

The final statement of the theorem is a consequence of the fact, 
remarked in \S\ref{S:deBR}, 
that $b$ and $\phi$ have the same inner factor.
\end{proof}

We next characterize those weights $\omega$ for which 
$\cD_\omega$ is equal to some de Branges--Rovnyak space $\cH(b)$. 
There is no harm in normalizing $\omega$ so that $\|\omega\|_{L^1(\DD)}=1$.
To state our result we need to introduce some notation.
Given  $\psi\in L^1(\DD)$, 
we denote by $Q\psi$ its \emph{Bergman projection}, given by
\[
Q\psi(z):=\int_\DD \frac{\psi(w)}{(1-\overline{w}z)^2}\,dA(w) 
\qquad(z\in\DD),
\]
and by $B\psi$ its \emph{Berezin transform}, defined by
\[
B\psi(z):=\int_\DD \frac{(1-|z|^2)^2}{|1-\overline{w}z|^4}\psi(w)\,dA(w)
\qquad(z\in\DD).
\]
For further information about these, we refer to \cite{HKZ00}.

\begin{theorem}
Let $\omega$ be a weight with $\|\omega\|_{L^1(\DD)}=1$.
Then there exists $b$ such that $\cD_\omega=\cH(b)$,
with equality of norms,  
if and only if:
\begin{itemize}
\item[\rm(i)] $\cD_\omega\subset H^2$,
\item[\rm(ii)] polynomials are dense in $\cD_\omega$, and
\item[\rm(iii)] $Q\omega$ is outer and satisfies
\begin{equation}\label{E:QBeqn}
(1-|z|^2)|Q\omega(z)|^2 =B\omega(z)
\qquad(z\in\DD).
\end{equation}
\end{itemize}
In this case  the associated $\phi$ is given by  
$\phi(z)=cz(Q\omega)(z)$, where $c\in\TT$.
\end{theorem}

\begin{proof}
First we prove the `only if'. 
If $\cD_\omega=\cH(b)$, then
clearly $\cD_\omega\subset H^2$.
Also, as remarked in the previous theorem,
$b$ is non-extreme, 
so by Theorem~\ref{T:nonextreme} polynomials are dense in $\cH(b)$,
and thus also in $\cD_\omega$. 
This proves (i) and (ii). 
For (iii), note that equality of norms implies 
$\|k_z^+\|_{H^2}^2=\cD_\omega(k_z)$ for all $z\in\DD$, 
where once again $k_z$ denotes the Cauchy kernel.
This yields the identity
\begin{equation}\label{E:phiomega}
\frac{|\phi(z)|^2}{1-|z|^2}
=\int_\DD\frac{|z|^2}{|1-\overline{z}w|^4}\omega(w)\,dA(w)
\qquad(z\in\DD).
\end{equation}
By Theorem~\ref{T:binner}, 
we have $\phi(z)=z\phi_o(z)$, 
where $\phi_o$ is outer. 
It follows that
\[
|\phi_o(z)|^2=\int_\DD\frac{1-|z|^2}{|1-\overline{z}w|^4}\omega(w)\,dA(w)
\qquad(z\in\DD).
\]
In particular $|\phi_o(0)|^2=\|\omega\|_{L^1(\DD)}=1$. 
Also, polarizing, we obtain
\[
\phi_o(z_1)\overline{\phi_o(z_2)}
=\int_\DD\frac{1-z_1\overline{z}_2}{(1-\overline{w}z_1)^2(1-\overline{z}_2w)^2}\omega(w)\,dA(w)
\qquad(z_1,z_2\in\DD).
\]
Setting $z_2=0$, we find that $\phi_o=c(Q\omega)$, where $|c|=1$. 
Hence $Q\omega$ is outer, 
and substituting this back into \eqref{E:phiomega} gives \eqref{E:QBeqn}. 
This establishes (iii),
and also shows that $\phi(z)=cz(Q\omega)(z)$.

Now we  prove the `if'. 
Suppose that $\omega$ satisfies (i), (ii) and (iii).
Define $\phi(z):=z(Q\omega)(z)$. 
By (iii), the function $\phi$ belongs to the Smirnov class $N^+$,
so it can be written as $\phi=b/a$ for some pair $(b,a)$. 
We claim that $\cD_\omega=\cH(b)$ with equality of norms. 
Property~(i) implies that the norm $\|\cdot\|_{\cD_\omega}$ is well-defined.
Property~(ii) implies that the Cauchy kernels $k_z~(z\in\DD)$ 
span a dense subspace of $\cD_\omega$
(as they do for $\cH(b)$). 
It thus suffices to establish equality of norms,
and by the same argument as in the proof of Theorem~\ref{T:Dmu=Hb},
it is enough to prove that 
$\cD_\omega(k_z)=\|k_z^+\|_{H^2}^2$ for all $z\in\DD$.
This boils down to showing that \eqref{E:phiomega} holds, 
which, with our definition of $\phi$,
is equivalent to equation \eqref{E:QBeqn} of Property~(iii).
This establishes our claim and completes the proof.
\end{proof}

The last theorem obviously begs the question as to 
which weights satisfy properties (i), (ii) and (iii).
In particular:

\begin{question}
Which weights $\omega$ obey the relation \eqref{E:QBeqn}? 
\end{question}

\noindent
According to Theorem~\ref{T:Dmu=Hb},
equation \eqref{E:QBeqn} is satisfied by the weights
\[
\omega_\zeta(z):=
\begin{cases}
\displaystyle\frac{1-|z|^2}{|\zeta-z|^2}, 
&\zeta\in\TT,\\
\displaystyle\log\Bigl|\frac{1-\overline{\zeta}z}{\zeta-z}\Bigr|\frac{2}{1-|\zeta|^2}, 
&\zeta\in\DD,
\end{cases}
\]
as  can also be verified by direct calculation. 
Are there any other solutions?

A possible source of examples are weights that can be expressed as 
the difference of two positive superharmonic functions. 
Such weights are given by the formula \eqref{E:murep}, 
where now $\mu$ is a finite \emph{signed} measure on $\overline{\DD}$.
Part of the argument of the proof of Theorem~\ref{T:converse} carries over to this case, 
showing that $\mu$ must still satisfy the moment relation \eqref{E:moment}. 
This raises another question:
 
 \begin{question}
Which signed measures $\mu$ on $\overline{\DD}$ satisfy the relation \eqref{E:moment}?
\end{question}

\noindent
Obviously this is the case if $\mu$ is a multiple of a Dirac measure. Are there any others?

\subsection*{Acknowledgment}
Part of this research was carried out 
during a Research-in-Teams meeting 
at the Banff International Research Station (BIRS). 
We thank BIRS for its hospitality.

\end{document}